\RequirePackage{fix-cm}
\documentclass[smallextended]{svjour3}       % onecolumn (second format)
\smartqed  % flush right qed marks, e.g. at end of proof
\usepackage{graphicx}
\usepackage{amsmath}
\usepackage[numbers]{natbib}
\usepackage{bm}
\usepackage[utf8]{inputenc}
\begin{document}

\title{Optimal majority threshold in a stochastic environment
}
\author{Vitaly Malyshev
}

\maketitle

\begin{abstract}
Within the model of social dynamics determined by collective decisions in a stochastic
environment (the ViSE model), we consider the case of a homogeneous society consisting of
classically rational economic agents. We obtain analytical expressions
for the optimal majority threshold as a function of the parameters of the environment,
assuming that the proposals are generated by means of random variables.
The cases of several specific distributions of these variables are considered in more detail.
\end{abstract}

\section{Introduction}
\label{intro}
    In \citeauthor{borzenko2006} (\citeyear{borzenko2006}), the ViSE (Voting in a Stochastic Environment) model (hereinafter, the $model$) has been proposed.
    Its simplest version describes a society that consists
    of $n$ classically rational economic agents %(homines economici, see O'Boyle (2009)),
    who are boundedly rational egoists
    (hereafter, $egoists$). Each of them maximizes their individual utility in every act of choice,
    which turns out to be the most profitable noncooperative strategy. Various cooperative and egoistic strategies within the
    model have been studied in \citeauthor{borzenko2006} (\citeyear{borzenko2006}), \citeauthor{cheb2006} (\citeyear{cheb2006}),
    \citeauthor{cheb2009} (\citeyear{cheb2009}), and \citeauthor{mal2017} (\citeyear{mal2017}), and altruistic strategies in \citeauthor{cheb2018} (\citeyear{cheb2018}).

        Each $participant$/$agent$ is characterized by the
    current value of individual $utility.$ %(which can also be interpreted as the value of individual $utility$ or $welfare$).
    A $proposal$ of the environment is a vector of proposed utility increments of the participants.
    A similar model with randomly generated proposals appeared in \citeauthor{compte2017} (\citeyear{compte2017}).
    The society can accept or reject every proposal by means of voting, i.e., choose $reform$ or $status$ $quo$. Each agent votes for those and only
    those proposals that increase his/her individual utility.
    A proposal is accepted and implemented, i.e., the participants' utilities are incremented in accordance with the proposal,
    if and only if the proportion of the society supporting this
    proposal is greater than a $strict$ $relative$ $voting$ $threshold$ $\alpha\in[-\frac{1}{n},1]$.
    Otherwise, all utilities remain unchanged. This voting procedure is called ``$\alpha$-majority''
    (cf. \citeauthor{nitzan1982} (\citeyear{nitzan1982}, \citeyear{nitzan1984}),
    \citeauthor{fels2001} (\citeyear{fels2001}), \citeauthor{baharad2019} (\citeyear{baharad2019}),
     and \citeauthor{oboyle2009} (\citeyear{oboyle2009})).

        The voting threshold $\alpha$ will also be called the $majority$ $threshold$ or, more precisely, the $acceptance$ $threshold$, since
    $\alpha < 0.5$ is allowed.

      The concept of proposal allows one to model potential changes that are beneficial for some agents and disadvantageous for others.
    As a result of the implementation of such a proposal, the utilities of some agents increase, while the utilities of
    others decrease.

        The proposals are stochastically generated by the $environment$ and put to a general vote over and over again. The subject of
    the study is the dynamics of the participants' utilities as a result of this process. A similar dynamic model proposed by A. Malishevski has been presented
    in \citeauthor{mirkin1979} (\citeyear{mirkin1979}), Subsection 1.3 of Chapter 2.         Another model whose simplest version is very close to the simplest version of the ViSE model was studied in \citeauthor{barb2006} (\citeyear{barb2006}).

        Some other related voting models have been studied in the theory of legislative bargaining (see \citeauthor{duggan2012} (\citeyear{duggan2012})),
    where stochastic generation of proposals has been assumed in some cases (\citeauthor{penn2009} (\citeyear{penn2009}), \citeauthor{dziuda2014} (\citeyear{dziuda2014}, \citeyear{dziuda2016})).
     On other connections between the ViSE model and various comparable models, we refer to \citeauthor{cheb2018} (\citeyear{cheb2018}).
    %Does it inevitably lead to an increase in the social welfare or can democratic decisions systematically reduce the total capital of the society? Does the financial inequality grow? How many participants are ruined?

        In accordance with the ViSE model, the utility increments/decrements that form proposals
    are realizations of independent identically distributed random variables (independence is taken as a base case,
    models with dependent or non-identically distributed random variables can also be considered). In this paper, we present a general
    result applicable to any distribution that has a mathematical expectation and focus on four families of
    distributions: continuous uniform distributions, normal distributions (cf. \citeauthor{opt2018} (\citeyear{opt2018})), symmetrized Pareto distributions (see \citeauthor{cheb2018} (\citeyear{cheb2018})), and Laplace distributions.

        Each distribution is characterized by its mathematical expectation, $\mu$ and standard deviation, $\sigma$.
    The ratio $\sigma/\mu$ is called the coefficient of variation of a random variable.
    The inverse coefficient of variation $\rho = \mu/\sigma$, which we call the $adjusted$ (or $normalized$) $mean$ $of$ $the$
    $environment$, measures the  relative favorability of the environment. If $\rho > 0$, then the opportunities provided by the environment
    are favorable on average; if $\rho < 0$, then the environment is unfavorable. We introduce the concept of optimal acceptance threshold and investigate dependence of this threshold on $\rho$ for several types of distributions.

    %An aspect of social practice that can be examined using the ViSE model is adoption by the Parliament of various bills that are "prompted by life" (environment), i.e., by economic and/or political conjuncture. It can be assumed that the members of Parliament, being lobbyists of certain interests and adherents of certain beliefs, are so interested in accepting or rejecting a bill (e.g., a budget draft) that this concernment is adequately expressed in terms of individual utility or capital. Of course, the adoption of any other collective decisions can also be considered (to some extent) in terms of the ViSE model and its modifications.
    In the present paper, we study: %\vspace{-1.75mm}
    \begin{itemize}
        \item the optimal acceptance threshold for a general distribution (Subsection \ref{sec:general}), i.e., the threshold that maximizes the social welfare (this generalizes Theorem 1 in \citeauthor{opt2018} (\citeyear{opt2018}));
        \item dependence of the optimal acceptance threshold on the model parameters for several specific distributions
        (Subsections \ref{sec:2} to \ref{sec:4}).
        \item expected utility increase for a general distribution (Section \ref{sec:expUtInc}) (this generalizes Lemma 1 in \citeauthor{cheb2006} (\citeyear{cheb2006})).
    \end{itemize}

\section{Optimal majority threshold}
\label{sec:1}

\subsection{The model}%A general result}
\label{sec:model}
    To familiarize with the problem that the optimal majority threshold solves, let us look at
    the dependence of the expected utility increment of an agent on the adjusted mean of the environment $\rho$ (\citeauthor{opt2018} (\citeyear{opt2018})).

    %One-step mean capital increment is an expected value of one agent's capital increment (hereafter $\eta$)
    %in one step taking into account the voting procedure. Let $\zeta$ be the proposed capital increment of
    %an agent (random variable e.g. normally distributed) and $\overline{\zeta}$ is a vector of proposed capitals
    %of all participants (includes $\zeta$ as component), then $\eta$ is $\zeta I(\overline{\zeta}, \alpha n)$, where
    Let $\bm\zeta=(\zeta_1,\ldots,\zeta_n)$ denote a random proposal on some step. Its component $\zeta_i$ is the proposed utility increment of agent~$i$. The components $\zeta_1,\ldots,\zeta_n $ are independent identically distributed random variables. $\zeta$~will denote a similar scalar variable without reference to a specific agent.
    Similarly, let $\bm\eta=(\eta_1,\ldots,\eta_n)$ be the random vector of \emph{actual\/} increments of the agents on the same step. If $\bm\zeta$ is adopted, then $\bm\eta=\bm\zeta$; otherwise $\bm\eta=(0,\ldots,0)$. Consequently,
        \begin{equation}\label{e:etazeta}
            \bm\eta=\bm\zeta I(\bm\zeta, \alpha n),
        \end{equation}
    where\footnote{$\#X$ denotes the number of elements in the finite set $X$.}
        \begin{eqnarray}\label{indicator}
            I(\bm\zeta, \alpha n) =
                        \begin{cases}
                            1, &\text{$\#\{k:\zeta_k > 0, k = 1,...,n\} > \alpha n$ } \\ %if the number of positive components of $\bm\zeta$ is greater than $\alpha n$;}\\
                            0, &\text{otherwise.}
                        \end{cases}
        \end{eqnarray}
    and $[\alpha n] \leq n$, $[\alpha n] = n-1$ corresponds to unanimity, $[\alpha n] = -1$ and $[\alpha n] = n$ to accept and reject of proposal without voting, respectively\footnote{$[\alpha n]$ is the integer part of $\alpha n$.}.

    Eq. \eqref{e:etazeta} follows from the assumption that each agent votes for those and only those proposals that increase his/her individual utility.

    Let $\eta$ be a random variable similar to every $\eta_i,$ but having no reference to a specific agent. We are interested in the expected utility increment of an agent, i.e. ${\rm E}(\eta),$ where ${\rm E}(\cdot)$ is the mathematical expectation.

    Consider an example. For 21 participants and $\alpha=0.5$, the dependence of ${\rm E}(\eta)$ on $\rho=\mu/\sigma$ is presented in Fig. \ref{fig:pit}, where proposals are generated by the normal distribution.

    \begin{figure}
          \includegraphics{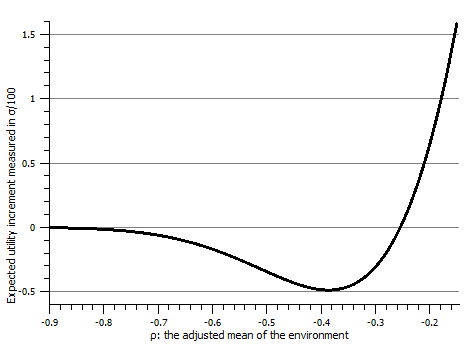}
        \caption{Expected utility increment of an agent: 21 agents; $\alpha = 0.5$; normal distribution.}
        \label{fig:pit}       % Give a unique label
    \end{figure}

    Fig. \ref{fig:pit} shows that for $\rho \in (-0.85,-0.266),$ the expected utility increment is an appreciable
    negative value, i.e., proposals approved by the majority are, on average, unprofitable and impoverishing for the society.
    This part of the curve is called a ``$pit$ $of$ $losses.$'' For $\rho < -0.85,$
    the negative mean increment is very close to zero, since the proposals are extremely rarely accepted.

\subsection{A voting sample}
\label{sec:voting_sample}

    By a ``$voting$ $sample$'' of size $n$ with absolute voting threshold $n_0$ we mean the vector of random variables $(\zeta_1 I(\bm\zeta, n_0),\ldots,\zeta_n I(\bm\zeta, n_0))$, where $\bm\zeta=(\zeta_1,\ldots,\zeta_n)$ is a sample from some distribution and $I(\bm\zeta, n_0)$ is defined by \eqref{indicator}.
        %\begin{eqnarray}\label{indicator}
         %   I(\bm\zeta, n_0) =
          %              \begin{cases}
           %                 1, &\text{if the number of positive components of $\bm\zeta$ is greater than $n_0$;}\\ \nonumber
            %                0, &\text{otherwise.} \nonumber
             %           \end{cases}
        %\end{eqnarray}
    According to this definition, a voting sample vanishes whenever the number of positive elements of sample $\bm\zeta$ does not exceed the threshold $n_0$.

    The lemma on ``normal voting samples'' obtained in \citeauthor{cheb2006} (\citeyear{cheb2006}) can be generalized
    as follows.
    \begin{theorem}
        \label{lemm:1}
        Let $\bm\eta=(\eta_1,\ldots,\eta_n)$ be a voting sample from some distribution with an absolute voting threshold $n_0 \in \{-1,0,...,n\}$. Then, for any $k=1,...,n,$
            \begin{eqnarray}\label{voting_sample}
                {\rm E}(\eta_k) = \sum\limits_{x=n_0+1}^{n}
                    \left((E^+ + E^-)\frac{x}{n} - E^-\right) \begin{pmatrix} n \\ x \end{pmatrix} p^x q^{n-x},
            \end{eqnarray}
       where $E^- = \big|{\rm E}(\zeta \; | \; \zeta \leq 0)\big|, E^+ = {\rm E}(\zeta \; | \; \zeta > 0),$
       $p = P\{\zeta > 0\} = 1 - F(0), q = P\{\zeta \leq 0\} = F(0),$
        $\zeta$ is the random variable that determines the utility increment of any agent in a random proposal, and $F(\cdot)$ is the cumulative distribution function of $\zeta.$
    \end{theorem}

    \begin{proof} According to the formula of total probability for mathematical expectations, we have
        \begin{eqnarray}\label{expectation_as_sum}
            {\rm E}(\eta_k) = 0 \cdot P\{n^+ \leq n_0\} \; + \sum\limits_{x=n_0+1}^{n}
                {\rm E}(\eta_k \; | \; n^+ = x) P\{n^+ = x\}
        \end{eqnarray}
        for any $k =1,...,n$ (hereafter we use $\eta$ instead of $\eta_k$). Furthermore,

        \begin{eqnarray}\label{e_sum}
                {\rm E}(\eta \; | \; n^+ = x) = {\rm E}(\eta \; | \; n^+ = x, \; \eta > 0) \, P\{\eta > 0 \; | \; n^+ = x \} \\
                    + {\rm E}(\eta \; | \; n^+ = x, \; \eta \leq 0) \, P\{\eta \leq 0 \; | \; n^+ = x \}, \nonumber
        \end{eqnarray}

        \begin{equation}\label{e_positive}
                P\{\eta > 0 \; | \; n^+ = x \} = \frac{x}{n}, \;
        \end{equation}

        and

        \begin{equation}\label{e_negative}
                P\{\eta \leq 0 \; | \; n^+ = x \} = 1 - \frac{x}{n}.
        \end{equation}

        Using the fact that by the
        independence of the components of the proposal, for all $x>n_0$ we have ${\rm E}(\eta \; | \; n^+ = x, \; \eta > 0) = {\rm E}(\eta \; | \; \eta > 0)
        = {\rm E}(\zeta \; | \; \zeta > 0) = E^+$
        and similarly $\big|{\rm E}(\eta \; | \; n^+ = x, \; \eta \leq 0)\big| = \big|{\rm E}(\eta \; | \; n^+>n_0, \; \eta \leq 0)\big|
        = \big|{\rm E}(\zeta \; | \; \zeta \leq 0)\big| = E^-$
        and substituting \eqref{e_sum}--\eqref{e_negative} and $P\{n^+=x\} = \begin{pmatrix} n \\ x \end{pmatrix} p^x q^{n-x}$
        (where $p$ is the probability that a proposal component is positive and
        $q = 1- p$) into \eqref{expectation_as_sum} we get \eqref{voting_sample}.
        \qed
    \end{proof}

    \begin{corollary}
        \label{corollary:zeroAlpha}
        Let $\bm\eta=(\eta_1,\ldots,\eta_n)$ be a voting sample from some distribution with an absolute voting threshold $n_0 = 0$. Then, for any $k=1,...,n,$
            \begin{eqnarray}\label{zero_voting_sample}
                {\rm E}(\eta_k) = \mu + E^-(1-p)^n,
            \end{eqnarray}
        where $\mu = {\rm E}(\zeta)$
        and the other notations are defined in Theorem \ref{lemm:1}.
    \end{corollary}

    \begin{proof}
        Using the properties of the binomial distribution we get

        \begin{align*} %\label{zero_proof}
            {\rm E}(\eta_k) &= \sum\limits_{x=1}^{n}
                    \left((E^+ + E^-)\frac{x}{n} - E^-\right) \begin{pmatrix} n \\ x \end{pmatrix} p^x q^{n-x} \\  \nonumber
            &= (E^+ + E^-) \sum\limits_{x=0}^{n} \frac{x}{n} \begin{pmatrix} n \\ x \end{pmatrix} p^x q^{n-x}
               - E^- \sum\limits_{x=1}^{n} \begin{pmatrix} n \\ x \end{pmatrix} p^x q^{n-x} \\  \nonumber
            &= (E^+ + E^-) p - E^- \left(1 - \begin{pmatrix} n \\ 0 \end{pmatrix} p^0 q^{n}\right)  \\ \nonumber
            &= p E^+ - (1-p) E^- + E^-(1-p)^n. \nonumber
        \end{align*}

        Since $p E^+ - (1-p) E^- = \mu,$ we have \eqref{zero_voting_sample}. \qed

    \end{proof}

\subsection{A general expression for the optimal voting threshold}
\label{sec:general}

    For each specific environment, there is an $optimal$ $acceptance$ $threshold$\footnote{See \citeauthor{nitzan1982} (\citeyear{nitzan1982}) and
    \citeauthor{azr2014} (\citeyear{azr2014}) on other approaches to optimizing the majority threshold and \citeauthor{rae1969} (\citeyear{rae1969}) and
    \citeauthor{sek2015} (\citeyear{sek2015}) for a discussion of the case of multiple voting in this context.}
    $\alpha_0$ that provides the highest possible
    expected utility increment ${\rm E}(\eta)$ of an agent.

    The optimal acceptance threshold for the normal distribution as a function of the environment parameters has been studied in
    \citeauthor{opt2018} (\citeyear{opt2018}). This threshold turns out to be independent of the size of the society $n$.

    Voting with the optimal acceptance thresholds always yields positive expected utility increments
    and so it is devoid of ``pits of losses.''

    The following theorem provides a general expression for the optimal voting threshold, which holds for any distribution
    that has a mathematical expectation.

    \begin{theorem}
        \label{the:1}
        In a society consisting of egoists, the optimal voting threshold is
            \begin{eqnarray}\label{optimal_threshold}
                \alpha_0 = \left(1+\frac{E^+}{E^-}\right)^{-1},
            \end{eqnarray}
        where $E^- = \big|{\rm E}(\zeta \; | \; \zeta \leq 0)\big|, E^+ = {\rm E}(\zeta \; | \; \zeta > 0),$ and $\zeta$ is the random variable that determines the utility increment of any agent in a random proposal$.$
    \end{theorem}

    In terms of the value $R = \frac{E^+}{E^-},$ which we call the $win/loss$ $magnitude$ $ratio,$ equation \eqref{optimal_threshold} takes the form

    \begin{eqnarray}\label{optimal_threshold_1}
        \alpha_0 = \left(1+R\right)^{-1}. \nonumber
    \end{eqnarray}

    %The proofs are given in the Appendix.
    \begin{proof}
        Consider the expected social welfare increase when some proposal is adopted:
        \begin{equation}\label{expectedSocWelf}
            n^+ E^+ - (n - n^+) E^-, \nonumber
        \end{equation}
        where $n^+$ is the number of positive components in a proposal.

        This expression is positive if and only if $\frac{n^+}{n} > \frac{E^-}{E^+ + E^-}.$

        We obtained an analytical expression for the expected utility increment as the sum \eqref{voting_sample} in Theorem \ref{lemm:1}.
        Let us consider the sign of the sum terms. They are positive if and only if $\frac{x}{n} > \frac{E^-}{E^+ + E^-}.$

        %The optimal threshold allows to take only positive terms getting the maximum sum.
        Therefore, the voting threshold $\alpha_0 = \frac{E^-}{E^+ + E^-} = \left(1+\frac{E^+}{E^-}\right)^{-1}$
        allows us to take exactly all positive terms getting the maximum sum.
        %results in adopting all proposals with positive expectation and declining all others.
        Consequently, this threshold is optimal for society and for each agent due to
        uniformity.
        \footnote{This result can also be obtained by applying Theorem 1 in \citeauthor{barb2006} (\citeyear{barb2006}) if we consider each agent as a country
        with $n_i = 1$ (population) and a simple voting behaviour of the representative.
        In this case, $\alpha_0$-majority maximizes social and individual
        welfare. In the proof of Theorem \ref{the:1}, we provide a simpler argument for the case under consideration. Theorem 1 in \citeauthor{azr2014} (\citeyear{azr2014})
        can also be used for this proof if we consider environment proposals (in the ViSE model) as agent types in their model.}
        %\footnote{There is a little more complicated proof that uses explicit form of expected utility increment from Section \ref{sec:expUtInc}}
        %proof of Lemma 1 in \citeauthor{cheb2006} (\citeyear{cheb2006}).}

        \qed

        Note that if a threshold $\alpha$ is optimal and $[\alpha_1 n]=[\alpha n],$ then $\alpha_1$ is also an optimal threshold.

    \end{proof}

    %The dependence of the optimal majority threshold on $\mu$ is a ``ladder'' with all steps of height $\frac1n.$

    Let $\bar\alpha_0$ be the center of the half-interval of optimal acceptance thresholds for
    fixed $n,\sigma$, and $\mu$. Then this half-interval is $[\bar\alpha_0-\tfrac1{2n}, \bar\alpha_0+\tfrac1{2n}[$.
    Figures \ref{fig:uniform} to \ref{fig:laplace} show the dependence of $\bar\alpha_0$ on $\rho=\mu/\sigma$
    for several distributions used for the generation of proposals.

    As one can observe for various distributions,
    outside the segment $\rho\in[-0.7,\, 0.7]$, if an acceptance threshold is close to the optimal one and the number of
    participants is appreciable, then the proposals are almost always accepted (to the right of the segment) or almost always
    rejected (to the left of this segment). Therefore, in these cases, the issue of determining the exact optimal
    threshold loses its practical value.

%Text with citations \cite{RefB} and \cite{RefJ}.
\subsection{Proposals generated by continuous uniform distributions}
\label{sec:2}

    Let $-a < 0$ and $b > 0$ be the minimum and maximum values of a continuous uniformly distributed random variable, respectively.

    \begin{corollary}
        \label{corollary:1}
            The optimal majority/acceptance threshold in the case of proposals generated by the continuous uniform distribution on the segment $[-a, b]$ with $-a<0$ and $b>0$ is
                \begin{eqnarray}\label{uniform}
                    \alpha_0 = \left(1+\frac{b}a\right)^{-1}.
                \end{eqnarray}
    \end{corollary}

    Indeed, in this case, $E^- = \frac{a}{2}, E^+ = \frac{b}{2},$
    and $ R = \frac{b}{a},$ hence, \eqref{optimal_threshold} provides \eqref{uniform}.

    \begin{figure}
      \includegraphics{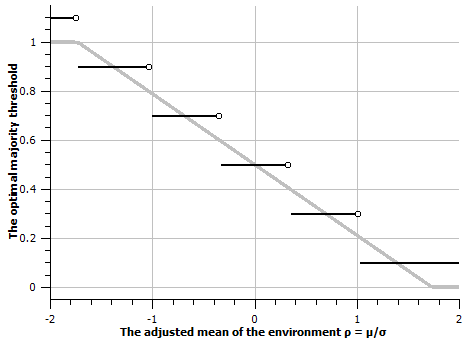}
    \caption{The center $\bar a_0$ of the half-interval of optimal majority/acceptance thresholds $($a ``ladder''$)$ for $n=5$ and the optimal threshold \eqref{uniform} as functions of $\rho$ for continuous uniform distributions.}
    \label{fig:uniform}       % Give a unique label
    \end{figure}

    If $b$ approaches 0 from above, then $\alpha_0$ approaches 1 from below, and the optimal voting procedure is unanimity. Indeed,
    positive proposed utility increments become much smaller in absolute value than negative ones, therefore, each participant
    should be able to reject a proposal.

    As $-a$ approaches 0 from below,
    negative proposed utility increments become much smaller in absolute value than positive ones. Therefore, a
    ``coalition'' consisting of any single voter should be able to accept a proposal. In accordance with this, the optimal relative threshold
     $\alpha_0$ decreases to 0.

    \begin{corollary}
        \label{corollary:2}
        In terms of the adjusted mean of the environment $\rho=\mu/\sigma,$ it holds that for the continuous uniform distribution,
            \begin{eqnarray}\label{uniform1}
                \alpha_0 =
                            \begin{cases}
                                1, &\text{$\rho \leq -\sqrt{3}$},\\
                                \frac{1}{2} \left(1 - \frac{\rho}{\sqrt{3}} \right), &\text{$-\sqrt{3} < \rho < \sqrt{3}$},\\
                                0, &\text{$\rho \geq \sqrt{3}$}.
                            \end{cases}
            \end{eqnarray}
    \end{corollary}
    This follows from \eqref{uniform} and the expressions $\mu = \frac{-a+b}{2}$ and $\sigma = \frac{b+a}{2\sqrt{3}}.$
    It is worth mentioning that the dependence of $\alpha_0$ on $\rho$ is linear, as distinct from \eqref{uniform}.

    Figure \ref{fig:uniform} illustrates the dependence of the center of the half-interval
    of optimal majority/acceptance thresholds versus $\rho=\mu/\sigma$ for continuous uniform distributions in the segment $\rho\in[-2,\, 2]$.

\subsection{Proposals generated by normal distributions}
\label{sec:3}

    \begin{figure}
      \includegraphics{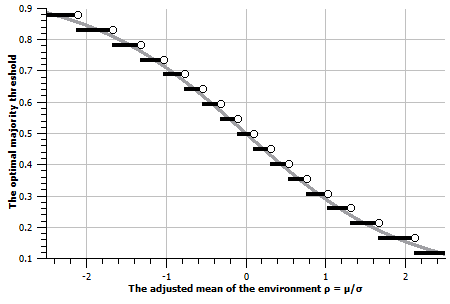}
    \caption{The center $\bar a_0$ of the half-interval of optimal majority/acceptance thresholds $($a ``ladder''$)$ for $n=21$ and the optimal threshold \eqref{normal} as functions of $\rho$ for normal distributions.}
    \label{fig:normal}       % Give a unique label
    \end{figure}

    For normal distributions, the following corollary holds.
    \begin{corollary}
        \label{corollary:3}
            The optimal majority/acceptance threshold in the case of proposals generated by the normal distribution with
            parameters $\mu$ and $\sigma$ is
                \begin{eqnarray}\label{normal}
                    \alpha_0 = F(\rho) \left( 1 - \frac{\rho F(-\rho)}{f(\rho)} \right),
                \end{eqnarray}
            where
            $\,\rho = \mu/\sigma,$
            while $\,F(\cdot)$ and $f(\cdot)\,$ are the standard normal cumulative distribution function and density, respectively.
    \end{corollary}

    Corollary \ref{corollary:3} follows from Theorem \ref{the:1} and the facts that $E^- = -\sigma \left(\rho - \frac{f(\rho)}{F(-\rho)}\right)$ and
    $E^+ = \sigma \left(\rho + \frac{f(\rho)}{F(\rho)}\right),$ which can be easily found by integration.  Note that Corollary \ref{corollary:3} strengthens the
    first statement of Theorem 1 in \citeauthor{opt2018} (\citeyear{opt2018}).

    Figure \ref{fig:normal} illustrates the dependence of the center of the half-interval
    of optimal majority/acceptance thresholds versus $\rho=\mu/\sigma$ for normal distributions in the segment $\rho\in[-2.5,\, 2.5]$.

    We refer to \citeauthor{opt2018} (\citeyear{opt2018}) for some additional properties (e.g., the rate of change of the
    optimal voting threshold as a function of $\rho$).

\subsection{Proposals generated by symmetrized Pareto distributions}
\label{sec:4}

    \begin{figure}
      \includegraphics{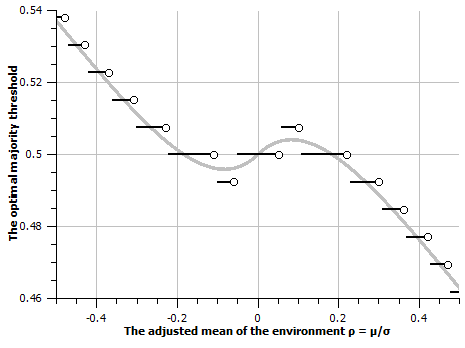}
    \caption{The center $\bar a_0$ of the half-interval of optimal majority/acceptance thresholds $($a ``ladder''$)$ for $n=131$ (odd) and the optimal threshold \eqref{pareto} as functions of $\rho$ for symmetrized Pareto distributions with $k$ = 8.}
    \label{fig:pareto}       % Give a unique label
    \end{figure}
    \begin{figure}
      \includegraphics{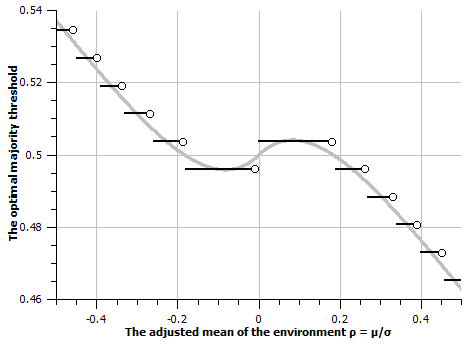}
    \caption{The center $\bar a_0$ of the half-interval of optimal majority/acceptance thresholds $($a ``ladder''$)$ for $n=130$ (even) and the optimal threshold \eqref{pareto} as functions of $\rho$ for symmetrized Pareto distributions with $k$ = 8.}
    \label{fig:pareto_even}       % Give a unique label
    \end{figure}

    Pareto distributions are widely used for modeling social, linguistic, geophysical, financial, and some other types of data.
    The Pareto distribution with positive parameters $k$ and $a$ can be defined by  means of the function $P\{\xi > x\} = \left(\frac{a}{x}\right)^{k},$ where $\xi \in [a,\infty)$ is a random variable.

    The ViSE model normally involves distributions that allow both positive and negative values. Consider the $symmetrized$
    $Pareto$ $distributions$ (see \citeauthor{cheb2018} (\citeyear{cheb2018}) for more details). For its construction, the density function
    $f(x) = \frac{k}{x} \left(\frac{a}{x}\right)^{k}$
    of the Pareto distribution is divided by 2 and combined with its reflection w.r.t. the line $x = a$.

    The density of the resulting distribution with mode (and median) $\mu$ is
    \begin{equation}\label{e_diff}
            f(x) = \frac{k}{2a} \left(\frac{|x-\mu|}{a} + 1\right)^{-(k+1)}. \nonumber
    \end{equation}

    For symmetrized Pareto distributions with $k>2,$ the following result holds true.
    \begin{corollary}
        \label{corollary:4}
            The optimal majority/acceptance threshold in the case of proposals generated by the symmetrized Pareto
            distribution with parameters $\mu,$ $\sigma,$ and $k>2$ is
                \begin{equation}\label{pareto}
                    \alpha_0 = \frac{1}{2}\left(1+{\rm sign}(\rho)\frac{1-(k-2)\hat\rho-(1+\hat\rho)^{-k+1}}{1+k\hat\rho}\right)
                        %\begin{cases}
                            %p \left( \frac{B}{B + q\mu} \right), &\text{$\mu > 0$}\\
                            %p \left( \frac{A - \mu}{A - p\mu} \right), &\text{$\mu \leq 0$},
                            % \frac{C + \rho}{k\rho + C} \left(
                            %1 - \frac{1}{2} \left( \frac{C}{C + \rho} \right)^k \right), &\text{$\mu > 0$},\\
                            % 1 + \frac{C-\rho}{k\rho - C} \left( 1 - \frac{1}{2} \left( \frac{C}{C - \rho} \right)^k
                            % \right), &\text{$\mu \leq 0$},
                        %\end{cases}
                \end{equation}
            where $\,\rho = \frac{\mu}{\sigma}$,
            $C = \sqrt{\frac{(k-1)(k-2)}{2}} = \frac{a}{\sigma}$, and $\hat\rho=|\rho/C|=|\mu/a|$.
            %$ a = \sigma \sqrt{\frac{(k-1)(k-2)}{2}}$ , %$b = a + \mu$,
            %$b' = a - \mu$, $A =  \left(\frac{a}{b'}\right)^k \frac{b'}{2(k-1)} $, $B =  \left(\frac{a}{b}\right)^k \frac{b}{2(k-1)} $.
    \end{corollary}

    Corollary \ref{corollary:4} follows from Theorem \ref{the:1} and the facts (their proof is given below) that:

    $E^- = \sigma \left(\frac{C + \rho}{k-1}\right),$
    $E^+ = \frac{\sigma}{1 - \frac{1}{2}\left( \frac{C}{C + \rho} \right)^k}
        \left(\rho + \left( \frac{C}{C + \rho} \right)^k
        \frac{C + \rho}{2(k-1)}\right)$
    whenever $\mu > 0;$

    $E^- = -\frac{\sigma}{1 - \frac{1}{2}\left( \frac{C}{C - \rho} \right)^k}
        \left(\rho - \left( \frac{C}{C - \rho} \right)^k
        \frac{C - \rho}{2(k-1)}\right),$
    $E^+ \; = \;  \sigma \left(\frac{C - \rho}{k-1}\right)$
    whenever $\mu \leq 0.$

    The ``ladder'' and the optimal acceptance threshold curve for symmetrized Pareto distributions are fundamentally different from
    the corresponding graphs for the normal and continuous uniform distributions. Namely, $\alpha_0(\rho)$ increases in some neighborhood of $\rho = 0$.
    % for odd $n$'s and an increase of $\alpha_0$ when $\rho$ becomes positive  for even $n$'s.

    %Correspondingly, the curve $\alpha_0(\rho)$ has an abnormal part in a vicinity of zero,
    %where the optimal threshold $increases$ with the adjusted mean.
    As a result, $\alpha_0(\rho)$ has two extremes.
    This is caused by the following peculiarities of the symmetrized Pareto distribution: an increase of $\rho$ from negative to positive values decreases
    $E^+$ and increases $E^-$. By virtue of \eqref{optimal_threshold}, this causes an increase of $\alpha_0.$

    This means that the plausible hypothesis about the profitability of the voting threshold raising when
    the environment becomes less favorable (while the type of distribution and $\sigma$ are preserved) is not generally true.
    In contrast, for symmetrized Pareto distributions, it is advantageous to lower the threshold whenever a decreasing $\rho$ remains close to zero (an abnormal part of the graph).

    Figures \ref{fig:pareto} and \ref{fig:pareto_even} illustrate the dependence of the center of the half-interval
    of optimal voting thresholds versus $\rho=\mu/\sigma$ for symmetrized Pareto distributions with $k=8.$

    \begin{proof} {$of$ $Corollary$ {\it \ref{corollary:4} } }
    Let $F(\cdot)$ and $f(\cdot)\,$ be the cumulative Pareto distribution function and the Pareto density, respectively;
    $\,\rho=\mu/\sigma,$ and $C \sigma = \sigma \sqrt{\frac{(k-1)(k-2)}{2}} = a. $

    Let $\mu > 0$. Then

    \begin{align*}
            E^- &= -\frac{1}{F(0)} \int_{-\infty}^{0} x \frac{k}{2 C \sigma}
            \left( \frac{-x + C \sigma + \rho \sigma }{C \sigma} \right)^{-(k+1)} dx \nonumber \\
            &= -\frac{1}{\frac{1}{2} \left(\frac{C}{C + \rho}\right)^k} \frac{k}{2 C \sigma}
            \left( \frac{(C \sigma)^{k+1} (-x + C \sigma + \rho \sigma)^{-k}
            (kx - C \sigma - \rho \sigma) }{(k-1)k} \right) \bigg|_{-\infty}^{0} \nonumber \\
            &= \sigma \left(\frac{C + \rho}{k-1}\right); \nonumber
            %\bigg|_A^B
    \end{align*}

    \begin{align*}
            E^+ &= \frac{1}{1 - F(0)} \int_{0}^{\mu} x \frac{k}{2 C \sigma}
            \left( \frac{-x + C \sigma + \rho \sigma }{C \sigma} \right)^{-(k+1)} dx \nonumber \\
            &+ \; \; \frac{1}{1 - F(0)} \int_{\mu}^{\infty} x \frac{k}{2 C \sigma}
            \left( \frac{x + C \sigma - \rho \sigma }{C \sigma} \right)^{-(k+1)} dx \nonumber \\
            &= \frac{1}{1 - \frac{1}{2} \left(\frac{C}{C + \rho}\right)^k} \frac{k}{2 C \sigma}
            \left( \frac{(C \sigma)^{k+1} (-x + C \sigma + \rho \sigma)^{-k}
            (kx - C \sigma - \rho \sigma) }{(k-1)k} \right) \bigg|_{0}^{\rho \sigma}  \nonumber \\
            &- \frac{1}{1 - \frac{1}{2} \left(\frac{C}{C + \rho}\right)^k} \frac{k}{2 C \sigma}
            \left( \frac{(C \sigma)^{k+1} (x + C \sigma - \rho \sigma)^{-k}
            (kx + C \sigma - \rho \sigma) }{(k-1)k} \right) \bigg|_{\rho \sigma}^{\infty} \nonumber \\
            &= \frac{\sigma}{1 - \frac{1}{2}\left( \frac{C}{C + \rho} \right)^k}
            \left(\rho + \left( \frac{C}{C + \rho} \right)^k
            \frac{C + \rho}{2(k-1)}\right). \nonumber
            %\bigg|_A^B
    \end{align*}

    Similarly,
    $E^- = -\frac{\sigma}{1 - \frac{1}{2}\left( \frac{C}{C - \rho} \right)^k}
        \left(\rho - \left( \frac{C}{C - \rho} \right)^k
        \frac{C - \rho}{2(k-1)}\right)$ and
    $E^+ \; = \;  \sigma \left(\frac{C - \rho}{k-1}\right)$ whenever $\mu \leq 0$.
    \qed

    %Then

    %\begin{align}
     %       M^+ = \frac{1}{1 - F(0)} \int_{0}^{\infty} x \frac{k}{2 C \sigma}
     %       \left( \frac{x + C \sigma - \rho \sigma }{C \sigma} \right)^{-(k+1)} dx \nonumber \\
     %       = \frac{1}{1 - 1 + \frac{1}{2} \left(\frac{C}{C - \rho}\right)^k} \frac{-k}{2 C \sigma}
     %       \left\{ \frac{(C \sigma)^{k+1} (x + C \sigma - \rho \sigma)^{-k}
     %       (kx + C \sigma - \rho \sigma) }{(k-1)k} \right\} \bigg|_{0}^{\infty} \nonumber \\
     %       = \sigma \left(\frac{C - \rho}{k-1}\right); \nonumber
            %\bigg|_A^B
    %\end{align}

    %\begin{align}
     %       M^- = \frac{1}{F(0)} \int_{-\infty}^{\mu} x \frac{k}{2 C \sigma}
     %       \left( \frac{-x + C \sigma + \rho \sigma }{C \sigma} \right)^{-(k+1)} dx \nonumber \\
     %       + \; \; \frac{1}{F(0)} \int_{\mu}^{0} x \frac{k}{2 C \sigma}
     %       \left( \frac{x + C \sigma - \rho \sigma }{C \sigma} \right)^{-(k+1)} dx \nonumber \\
     %       = \frac{1}{1 - \frac{1}{2} \left(\frac{C}{C - \rho}\right)^k} \frac{k}{2 C \sigma}
     %       \left\{ \frac{(C \sigma)^{k+1} (-x + C \sigma + \rho \sigma)^{-k}
     %       (kx - C \sigma - \rho \sigma) }{(k-1)k} \right\} \bigg|_{-\infty}^{\rho \sigma}  \nonumber \\
      %      - \frac{1}{1 - \frac{1}{2} \left(\frac{C}{C - \rho}\right)^k} \frac{k}{2 C \sigma}
      %      \left\{ \frac{(C \sigma)^{k+1} (x + C \sigma - \rho \sigma)^{-k}
      %      (kx + C \sigma - \rho \sigma) }{(k-1)k} \right\} \bigg|_{\rho \sigma}^{0} \nonumber \\
      %      = \frac{\sigma}{1 - \frac{1}{2}\left( \frac{C}{C - \rho} \right)^k}
      %      \left\{\rho - \left( \frac{C}{C - \rho} \right)^k
      %      \frac{C - \rho}{2(k-1)}\right\}. \nonumber
            %\bigg|_A^B
    %\end{align}

    \end{proof}

\subsection{Proposals generated by Laplace distributions}
\label{sec:laplace}

    \begin{figure}
      \includegraphics{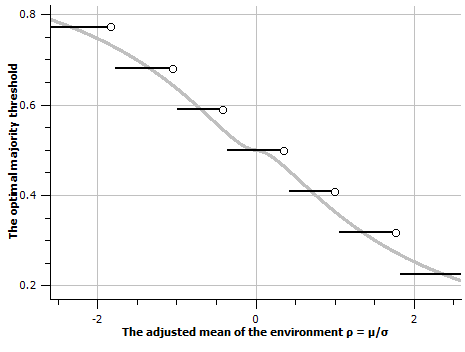}
    \caption{The center $\bar a_0$ of the half-interval of optimal majority/acceptance thresholds $($a ``ladder''$)$ for $n=11$ and the optimal threshold \eqref{laplace} as functions of $\rho$ for Laplace distributions.} % with $\lambda = 1$ $(\sigma = \sqrt{2}$).}
    \label{fig:laplace}       % Give a unique label
    \end{figure}

    The density of the Laplace distribution with parameters $\mu$ (location parameter)  and $\lambda > 0$ (rate parameter) is
    \begin{equation}\label{e_diff}
            f(x) = \frac{\lambda}{2} \exp{\left(-\lambda|x-\mu|\right)}. \nonumber
    \end{equation}

    For Laplace distributions, the following corollary holds.
    \begin{corollary}
        \label{corollary:5}
            The optimal majority/acceptance threshold in the case of proposals generated by the Laplace distribution with
            parameters $\mu$ and $\lambda$ is
                \begin{eqnarray}\label{laplace}
                    \alpha_0=\frac{1}{2}\left(1+{\rm sign}(\rho)\frac{1-\sqrt{2}|\rho|-\exp{\left(-\sqrt{2}|\rho|\right)}}{1+\sqrt{2}|\rho|}\right).
                    %\alpha_0 =
                     %   \begin{cases}
                      %      \frac{2-e^{-\lambda\mu}}{2 + 2\mu\lambda}, &\text{$\mu > 0$},\\
                       %     \frac{e^{\lambda\mu} - 2\mu\lambda}{2 - 2\mu\lambda}, &\text{$\mu \leq 0$}.
                        %\end{cases}
                \end{eqnarray}
    %where $\beta=|\lambda\mu|=|\lambda\sigma\rho| = |\sqrt{2}\rho|.$
    \end{corollary}

    Corollary \ref{corollary:5} follows from Theorem \ref{the:1} and the facts (their proof is similar to the proof of Corollary \ref{corollary:4}) that:

    $E^- = \frac{1}{\lambda},$
    $E^+ = \frac{2\mu + \frac{e^{-\lambda\mu}}{\lambda}}{2-e^{-\lambda\mu}}$
    whenever $\mu > 0;$

    $E^- = -\frac{2\mu - \frac{e^{\lambda\mu}}{\lambda}}{2-e^{\lambda\mu}},$
    $E^+ = \frac{1}{\lambda}$
    whenever $\mu \leq 0.$

    In Lemma 3 of \citeauthor{cheb2018} (\citeyear{cheb2018}), it was proved that the symmetrized Pareto distribution with parameters
    $k$, $\mu$, and $\sigma$ tends, as $k \to \infty$, to the Laplace distribution with the same mean and standard deviation.

    Notice that the abnormal part of the curve $\alpha_0(\rho)$ for symmetrized Pareto distributions becomes smaller with the growth of $k$.
    Figure 6 shows that it vanishes in the case of Laplace distribution. Now we verify this using the first derivative of $\alpha_0$ with respect to $\rho$. It is

   % \begin{proposition}\label{prop:derivative}
%        The first derivative of $\alpha_0$ with respect to $\mu$ for the Laplace distribution is
        \begin{eqnarray}\label{derivative}
                \frac{d\alpha_0}{d\rho} =
                %\sigma\frac{d\alpha_0}{d\mu} = \sigma \frac{\lambda e^{-\lambda |\mu|} (1 + \frac{1}{2} \lambda |\mu|) - \lambda}{(1+\lambda |\mu|)^2} =
                \frac{ e^{-\sqrt{2}|\rho|} (\sqrt{2} + |\rho|) - \sqrt{2}}{(1+\sqrt{2}|\rho|)^2}.
        \end{eqnarray}
%        where $\beta=|\lambda\mu|=|\lambda\sigma\rho| = |\sqrt{2}\rho|.$
%    \end{proposition}

    Note that this derivative is non-positive and is equal to zero only when $\rho = 0.$         The increasing part of the curve $\alpha_0(\rho)$ for the symmetrized Pareto distribution  degenerates to a single point as this distribution converges to the Laplace distribution. It is the point, where the first derivative \eqref{derivative} of $\alpha_0$ is equal to zero.

    \begin{figure}
      \includegraphics{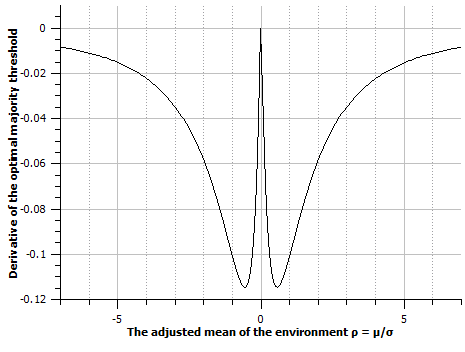}
    \caption{The first derivative of $\alpha_0$ with respect to $\rho$ \eqref{derivative} as a function of $\rho$ for Laplace distributions.}% with $\lambda = 1$ $(\sigma = \sqrt{2}$).}
    \label{fig:laplace_derivative}       % Give a unique label
    \end{figure}

    Figures \ref{fig:laplace} and \ref{fig:laplace_derivative} show the dependence of the center of the half-interval
    of optimal voting thresholds versus $\rho=\mu/\sigma$ and the dependence of the first derivative of $\alpha_0$ with respect to $\rho$ versus $\rho$ for Laplace distributions,
    respectively.

    We summarize the results of the above corollaries in Tables \ref{tab:prob_summary} and \ref{tab:win_loss_summary}.

    \begin{table}[!h]
    \caption{\label{tab:prob_summary}Probabilities of positive and negative proposals for several distributions.}
        \begin{center}
            \begin{tabular}{|l|c|c|c|c|c|}
                \hline
                {\centering              Distribution \par} & Parameters & $p$ & $q$ \\
                \hline
                Continuous uniform distribution & $-a < 0,b > 0$ & $\frac{b}{a+b}$ & $\frac{a}{a+b}$ \\
                Normal distribution & $\mu, \sigma$ & $F(\rho)$ & $F(-\rho)$  \\
                Symmetrized Pareto distribution $(\mu>0)$ & $k>2, \mu, \sigma$ & $1 - \frac{1}{2}\left( \frac{C}{C + \rho} \right)^k $
                & $\frac{1}{2}\left( \frac{C}{C + \rho} \right)^k$ \\
                Symmetrized Pareto distribution $(\mu \leq 0)$ & $k>2, \mu, \sigma$ & $\frac{1}{2}\left( \frac{C}{C - \rho} \right)^k$
                & $1 - \frac{1}{2}\left( \frac{C}{C - \rho} \right)^k $\\
                Laplace distribution $(\mu>0)$ & $\mu, \lambda$ & $1-\frac{1}{2}e^{-\lambda\mu}$ & $\frac{1}{2}e^{-\lambda\mu}$ \\
                Laplace distribution $(\mu \leq 0)$ & $\mu, \lambda$ & $\frac{1}{2}e^{\lambda\mu}$ & $1-\frac{1}{2}e^{\lambda\mu}$ \\
                \hline
            \end{tabular}
        \end{center}
        where $C = \sqrt{\frac{(k-1)(k-2)}{2}}$, $\,\rho = \mu/\sigma,$
            while $\,F(\cdot)$ is the standard normal cumulative distribution function.
    \end{table}

    \begin{table}[!h]
    \caption{\label{tab:win_loss_summary}Expected win and loss for several distributions.}
        \begin{center}
            \begin{tabular}{|l|c|c|c|c|c|}
                \hline
                Distribution & Parameters & $E^+$ & $E^-$ \\
                \hline
                Continuous uniform distribution & $-a < 0,b > 0$ & $\frac{b}{2}$ & $\frac{a}{2}$  \\
                Normal distribution & $\mu, \sigma$
                & $\mu + \sigma \frac{f(\rho)}{F(\rho)}$ &
                 $-\mu + \sigma \frac{f(\rho)}{F(-\rho)}$  \\
                Symmetrized Pareto distribution $(\mu>0)$ & $k>2, \mu, \sigma$
                & $\frac{\sigma}{p}
                    \left(\rho + q
                    \frac{C + \rho}{k-1}\right)$ & $\sigma \left(\frac{C + \rho}{k-1}\right)$\\
                Symmetrized Pareto distribution $(\mu \leq 0)$ & $k>2, \mu, \sigma$ & $\sigma \left(\frac{C - \rho}{k-1}\right)$
                & $-\frac{\sigma}{q}
                    \left(\rho - p
                    \frac{C - \rho}{k-1}\right)$\\
                Laplace distribution $(\mu>0)$ & $\mu, \lambda$
                 & $\frac{1}{p}\left(\mu + \frac{e^{-\lambda\mu}}{2\lambda}\right)$ & $\frac{1}{\lambda}$ \\
                Laplace distribution $(\mu \leq 0)$ & $\mu, \lambda$
                 & $\frac{1}{\lambda}$ & $-\frac{1}{q}\left(\mu - \frac{e^{\lambda\mu}}{2\lambda}\right)$ \\
                \hline
            \end{tabular}
        \end{center}
        where $C = \sqrt{\frac{(k-1)(k-2)}{2}}$, $\,\rho = \mu/\sigma,$
            while $\,F(\cdot)$ and $f(\cdot)\,$ are the standard normal cumulative distribution function and density, respectively;
            $p$ and $q$ are presented in the corresponding rows of Table \ref{tab:prob_summary}.
    \end{table}

\section{Expected utility increment}
\label{sec:expUtInc}
%\subsection{``Voting sample''}

    Let $Y_n \sim$ Bin$(n,p)$, where Bin$(n,p)$ is the binomial distribution with parameters $n$ and $p.$ Let
    $F_n(k) = \sum\limits_{x=0}^{k} \begin{pmatrix} n \\ x \end{pmatrix} p^x q^{n-x}$ be the cumulative distribution function of $Y_n$.
    Let $G_n(k) = 1 - F_n(k) = \sum\limits_{x=k+1}^{n} \begin{pmatrix} n \\ x \end{pmatrix} p^x q^{n-x}$. According to
    a relationship between the binomial distribution and the Beta distribution we have
        \begin{equation}\label{eq_beta}
                G_n(k) = B(p \; | \; k+1, n - k),
        \end{equation}
    where $B( \cdot \; | \; k + 1, n - k)$ is the cumulative distribution function of Beta distribution with $k+1$ and $n-k$ degrees of freedom.
    We also put $G_n(m) = 1$, when $m < 0$.

    Now we can prove the following theorem.

    \begin{theorem}
        \label{the:2}
        Let $\bm\eta=(\eta_1,\ldots,\eta_n)$ be a voting sample from some distribution with an absolute voting threshold\, $n_0 \in \{1,...,n-1\}$.
        Then for any $k=1,...,n$ it holds that
            \begin{eqnarray}\label{voting_sample_final}
                {\rm E}(\eta_k) = p (E^+ + E^-) B(p \; | \; n_0, n - n_0) - E^- B(p \; | \; n_0 + 1, n - n_0),
            \end{eqnarray}
        where $E^- = \big|{\rm E}(\zeta \; | \; \zeta \leq 0)\big|, E^+ = {\rm E}(\zeta \; | \; \zeta > 0),$
        $B(\cdot \; | \; m, l)$ is the cumulative distribution function of Beta distribution with $m$ and $l$ degrees of freedom,
        $p = P\{\zeta > 0\} = 1 - F(0),$
        $\zeta$ is the random variable that determines the utility increment of any agent in a random proposal, and $F(\cdot)$ is the cumulative distribution function of $\zeta.$
    \end{theorem}

    Corollary \ref{corollary:zeroAlpha}  in Subsection \ref{sec:voting_sample} extends the result of Theorem \ref{the:2} to $n_0=0$. It is easy to prove that ${\rm E}(\eta_k) = \mu$
    for $n_0=-1$ and ${\rm E}(\eta_k) = 0$ for $n_0=n.$

    \begin{proof}
       Let us prove the following equation:
            \begin{eqnarray}\label{exp}
                \sum\limits_{x=k}^{n} x \begin{pmatrix} n \\ x \end{pmatrix} p^x q^{n-x} = np \cdot G_{n-1}(k-2).
            \end{eqnarray}

       Denoting $y=x-1$ we have

        \begin{align*}\ %label{exp_proof}
            \sum\limits_{x=k}^{n} x \begin{pmatrix} n \\ x \end{pmatrix} p^x q^{n-x}  \nonumber
            &= np \sum\limits_{x=k}^{n} x \frac{(n-1)!}{(n-x)!x!} p^{x-1} q^{n-x}  \\ \nonumber
            &= np \sum\limits_{x=k}^{n} \frac{(n-1)!}{((n-1)-(x-1))!(x-1)!} p^{x-1} q^{(n-1)-(x-1)}  \\ \nonumber
            &= np \sum\limits_{x=k}^{n} \begin{pmatrix} n-1 \\ x-1 \end{pmatrix} p^{x-1} q^{(n-1)-(x-1)} \\ \nonumber
            &= np \sum\limits_{y=k-1}^{n-1} \begin{pmatrix} n-1 \\ y \end{pmatrix} p^{y} q^{(n-1)-y} \\ \nonumber
            &= np \cdot G_{n-1}(k-2). \nonumber
        \end{align*}

       Now we get \eqref{voting_sample_final} applying the definition of $G_n(k)$, \eqref{eq_beta}, and \eqref{exp} to \eqref{voting_sample}.
    \qed
    \end{proof}

        The formula \eqref{voting_sample_final} can be rewritten in terms of the regularized incomplete beta function.
        \footnote{Corollary \ref{corollary:incomplete_beta} has been suggested by an anonymous  referee.}

    \begin{corollary}
        \label{corollary:incomplete_beta}
        Let $\bm\eta=(\eta_1,\ldots,\eta_n)$ be a voting sample from some distribution with an absolute voting threshold\, $n_0 \in \{1,...,n-1\}$.
        Then for any $k=1,...,n$ it holds that
            \begin{eqnarray}\label{incomplete_beta_voting_sample}
                {\rm E}(\eta_k) = \mu I_p(n_0, n - n_0) + \frac{E^- p^{n_0} (1-p)^{n-n_0}}{n_0 B(n_0,n-n_0)},
            \end{eqnarray}
        where $I_p(n_0, n - n_0)$ and $B(n_0,n-n_0)$ are the regularized incomplete beta function and beta function, respectively,
        $\mu = {\rm E}(\zeta),$
        and the other notations are defined in Theorem \ref{the:2}.
    \end{corollary}

    \begin{proof}
        Using the properties of the Beta distribution and incomplete beta function we get

        \begin{align*} % \label{incomplete beta_proof}
            {\rm E}(\eta_k) &= p (E^+ + E^-) B(p \; | \; n_0, n - n_0) - E^- B(p \; | \; n_0 + 1, n - n_0) \\  \nonumber
            &= p (E^+ + E^-) I_p(n_0, n - n_0) - E^- \left(I_p(n_0, n - n_0) - \frac{p^{n_0} (1-p)^{n-n_0}}{n_0 B(n_0, n - n_0)} \right) \\  \nonumber
            &= \left( p E^+ - (1-p) E^- \right) I_p(n_0, n - n_0) + \frac{E^- p^{n_0} (1-p)^{n-n_0}}{n_0 B(n_0,n-n_0)} \\  \nonumber
            &= \mu I_p(n_0, n - n_0) + \frac{E^- p^{n_0} (1-p)^{n-n_0}}{n_0 B(n_0,n-n_0)}. \nonumber
        \end{align*}  \qed

    \end{proof}

    Theorem \ref{the:2} allows one to obtain a specific expression for ${\rm E}(\eta_k$) for each distribution generating proposals by applying specific forms of $E^+$, $E^-$, $p$, and $q$ (Tables \ref{tab:prob_summary} and \ref{tab:win_loss_summary}). For example, for continuous uniform distributions, the following corollary holds.

    \begin{corollary}
        \label{corollary:uniform}
            Let $\bm\eta=(\eta_1,\ldots,\eta_n)$ be a voting sample from the continuous uniform distribution on the segment $[-a, b]$ with $-a<0$ and $b>0$ with an absolute voting threshold $n_0 = \in \{1,...,n-1\}$. Then for any $k=1,...,n$ it holds that
                \begin{eqnarray}\label{cor_unif}
                    {\rm E}(\eta_k) = \frac{b}{2} B\left(\frac{b}{a+b} \; \middle\vert \; n_0, n - n_0\right) - \frac{a}{2} B\left(\frac{b}{a+b} \; \middle\vert \; n_0 + 1, n - n_0\right),
                \end{eqnarray}
            where $B(\cdot \; | \; m, l)$ is the cumulative distribution function of Beta distribution with $m$ and $l$ degrees of freedom.
    \end{corollary}

%\subsection{Continuous uniform and symmetrized Pareto distributions}

\section{Comparison of the expected utility increments}

In \citeauthor{cheb2018} (\citeyear{cheb2018}), the issue of correct location-and-scale standardization of distributions for the analysis of the ViSE model has been discussed.
An alternative (compared to using the same mean and variance) approach to standardizing continuous symmetric distributions was proposed. Namely, distributions similar in position and scale must have the same $\mu$ and the same interval (centered at $\mu$) containing a certain
essential proportion of probability. Such a standardization provides more similarity in the central region and the same weight of tails outside this region.

In what follows, we apply this approach for the comparison of the expected utility for several distributions. Namely, for each distribution, we find the variance such
that the first quartiles (and thus, all quartiles because the distributions are symmetric) coincide for zero mean distributions, where  the first quartile, $Q_1$, splits off the
``left'' 25\% of probability from the ``right'' 75\%.

For the normal distribution, $Q_1 \approx -0.6745 \sigma_N$, where $\sigma_N$ is
the standard deviation.

For the continuous uniform distribution, $Q_1 = -\frac{\sqrt{3}}{2} \sigma_U$, where $\sigma_U$ is its standard deviation.

For the symmetrized Pareto distribution, $Q_1 = C(1-2^\frac{1}{k})\sigma_P$, where $\sigma_P$ is the standard deviation and $C = \sqrt{\frac{(k-1)(k-2)}{2}}$.
This follows from the equation
        \begin{equation}\label{e_positive}
                F_P(Q_1) = \frac{1}{2}\left(\frac{C}{C-\frac{Q_1}{\sigma_P}}\right)^{k} = \frac{1}{4}, \nonumber
        \end{equation}
        where $F_P(\cdot)$ is the corresponding cumulative distribution function.

For the Laplace distribution, $Q_1 = \frac{-\ln{2}}{\lambda} = -\sigma_L\frac{\ln{2}}{\sqrt{2}}$, where $\sigma_L$ is the standard deviation.
%This follows from the equation
     %   \begin{equation}\label{e_positive}
  %              F_L(Q_1) = \frac{1}{2}\exp{(\lambda Q_1)} = \frac{1}{4}, \nonumber
  %      \end{equation}
  %      where $F_L(\cdot)$ is the cumulative Laplace distribution function.

Consequently, $\sigma_U \approx 0.7788\sigma_N$, $\sigma_P \approx 1.6262\sigma_N$ for $k = 8$, and $\sigma_L \approx 1.3762\sigma_N$.

    \begin{figure}
      \includegraphics{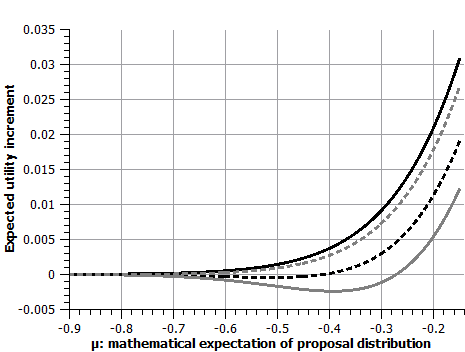}
    \caption{Expected utility increment of an agent as a function of $\mu$ with a majority/acceptance threshold of $\alpha = \frac{11}{21}$ for
      several distributions: black line denotes the symmetrized Pareto distribution, black dotted line the normal distribution (with $\sigma_N=1$), gray line the continuous
      uniform distribution, and gray dotted line the Laplace distribution.}
    \label{fig:expectedUtility}       % Give a unique label
    \end{figure}

    \begin{figure}
      \includegraphics{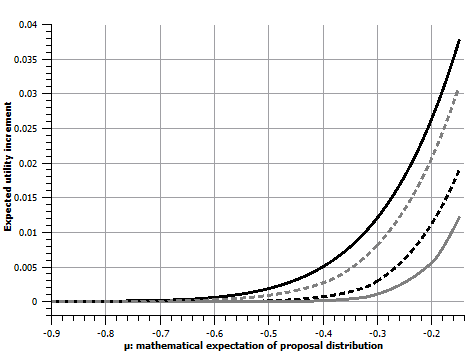}
    \caption{Expected utility increment of an agent as a function of $\mu$ when the optimal majority/acceptance thresholds are used for several distributions: black line denotes the symmetrized Pareto distribution, black dotted line the normal distribution (with $\sigma_N=1$), gray line the continuous
      uniform distribution, and gray dotted line the Laplace distribution.}
    \label{fig:optimalUtility}       % Give a unique label
    \end{figure}

    Figures \ref{fig:expectedUtility} and \ref{fig:optimalUtility} show the dependence of the expected utility increment of an agent on the mean $\mu$ of the proposal
    distribution for several distributions (normal, continuous uniform, symmetrized Pareto, and Laplace) for the majority threshold $\alpha = \frac{11}{21}$ and the optimal acceptance threshold, respectively.
    They are obtained by substituting the parameters of the environments into \eqref{voting_sample_final}, \eqref{uniform1}, \eqref{normal}, \eqref{pareto}, and \eqref{laplace}.
    Obviously, the optimal acceptance threshold excludes ``pits of losses'' because the
    society has the option to take insuperable
    threshold of 1 and reject all proposals.

    \begin{figure}
      \includegraphics{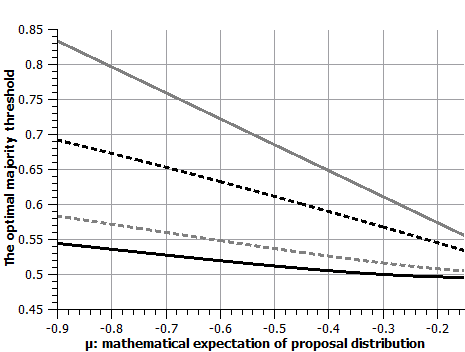}
    \caption{The optimal majority/acceptance threshold as function of $\mu$ for several distributions: black line denotes the symmetrized Pareto distribution, black dotted line the normal distribution, gray line the continuous
      uniform distribution, and gray dotted line the Laplace distribution.}
    \label{fig:optCompare}       % Give a unique label
    \end{figure}

    Figure \ref{fig:optCompare} illustrates the dependence of the optimal majority threshold on $\mu$ for the same list of distributions. It helps to explain why for $\alpha = \frac{11}{21},$ the continuous uniform distribution has the deepest pit of losses
     (because of the biggest difference between the actual and optimal thresholds), and why the symmetrized Pareto and Laplace distributions have no discernible pit of losses (because those differences are the smallest).

\section{Conclusion}
\label{conclusion}
    In this paper, we obtained general expressions for the expected utility increase and the optimal voting threshold
(i.e., the threshold that maximizes social/individual welfare) as functions of the parameters of the stochastic
proposal generator in the assumptions of the ViSE model. These expressions were given more specific forms for several types of distributions.

Estimation of the optimal majority/acceptance threshold seems to be a solvable problem in real situations.
If the model is at least approximately adequate and one can estimate the type of distribution and $\rho=\mu/\sigma$ by means of experiments,
then it is possible to obtain an estimate for the optimal acceptance threshold using the formulas provided in this paper.

We found that for some distributions of proposals, the plausible hypothesis
that it is beneficial to increase the voting threshold when the environment becomes less favorable is not generally true.
A deeper study of this issue should be the subject of future research.

%as required. Don't forget to give each section
%and subsection a unique label (see Sect.~\ref{sec:1}).
%\paragraph{Paragraph headings} Use paragraph headings as needed.
%\begin{equation}
%\end{equation}

%\begin{figure*}
%  \includegraphics[width=0.75\textwidth]{figures/norm21.png}
%\end{figure*}
%\begin{table}
%\begin{tabular}{lll}
%\hline\noalign{\smallskip}
%\end{tabular}
%\end{table}

%\begin{acknowledgements}
%\end{acknowledgements}

\bibliographystyle{spbasic}      % basic style, author-year citations
\bibliography{opt}   % name your BibTeX data base

\end{document}